\documentclass[12pt]{amsart}
\usepackage{amsmath,amsthm, amssymb,subeqnarray}
\usepackage{graphicx,epic}
\usepackage{mathrsfs,stackrel}
\usepackage{hyperref,color}
\usepackage{bookmark}

\usepackage{float}
\usepackage[utf8]{inputenc} 
\newtheorem{theorem}{Theorem}
\newtheorem{remark}{Remark}
\newtheorem{proposition}{Proposition}[section]
\newtheorem{lemma}[proposition]{Lemma}

\newtheorem{prop}{Facts}

\input{macros.sty}
\newcommand{\eg}{\emph{e.g.}\ }
\newcommand{\CI}{\mathcal{I}}
\begin{document}

\title{Generalized
Curie-Weiss model and quadratic pressure in ergodic theory}

\author{Renaud Leplaideur}
\address{ISEA,
Universit\'e de Nouvelle Calédonie,
145, Avenue James Cook - BP R4
         98 851 - Nouméa Cedex. Nouvelle Cal\'edonie}
    
\address{ LMBA, UMR6205,
Universit\'e de Brest.}
\email{Renaud.Leplaideur@univ-nc.nc,
\url{http://pagesperso.univ-brest.fr/~leplaide/} }

\author{Frédérique Watbled} 
\address{LMBA, UMR 6205, Universit\'e de Bretagne Sud,  Campus de Tohannic, BP 573, 56017 Vannes, France.}
\email{frederique.watbled@univ-ubs.fr}

\date{Version of \today}
\thanks{F. Watbled
 thanks the IRMAR, CNRS UMR 6625, University of Rennes 1, for its hospitality.}
\thanks{The authors thank the Centre Henri Lebesgue ANR-11-LABX-0020-01 for creating an attractive mathematical environment}

\begin{abstract}
We explain the Curie Weiss  model in Statistical Mechanics within the Ergodic viewpoint. 
More precisely,  we simultaneously define in $\{-1,+1\}^{\N}$, on the one hand a generalized Curie Weiss model within the Statistical Mechanics viewpoint and on the other hand, quadratic free energy and quadratic pressure  within the Ergodic Theory viewpoint. We show that there are finitely many invariant measures which maximize the quadratic free energy. They are all Dynamical Gibbs Measures. 
Moreover, the Probabilistic Gibbs measures for generalized Curie Weiss model converge to a determined combination of the (dynamical) conformal measures associated to these Dynamical Gibbs Measures. 
The standard Curie Weiss model is a particular case of our generalized Curie Weiss model. 
An Ergodic viewpoint over the Curie Weiss Potts model is also given. 
 \end{abstract}

\maketitle

\section{Introduction}\label{sec:intro}
\subsection{Background, main motivations and results}

The notion of Gibbs measure comes from Statistical Mechanics. It has been studied a lot from the probabilistic viewpoint (see \cite{Georgii, Costeniuc-Ellis-Touchette,Ellis-NewmanI, Ellis-Newman-Stat}). This notion was introduced in Ergodic Theory in the 70's by Sinai, Ruelle and Bowen (see \cite{sinai, sinaitransi, ruelle, ruelletransi, bowen}). Since that moment, the thermodynamic formalism became in Dynamical Systems a purely mathematical question and has somehow become isolated from the original physical questions.

With years, it turned out that this situation has generated sources of confusions. The first one is that people share the same vocabulary but it is not clear that the same names precisely define the same notions in each viewpoint (ergodic vs physicist). We \eg refer to \emph{phase transition, Gibbs measures, pressure}.  
Furthermore, the confusion is also internal to Ergodic Theory. Indeed, the thermodynamic formalism is very differently presented  for $\Z$-actions (where the Transfer Operator plays a crucial role) or for $\Z^{d}$-actions (with $d>1$). For this later case, the thermodynamic formalism is much closer to what people in Statistical Mechanics or in Probability do. Now, it turns out that several questions arising for 1-dimensional actions ergodic theory have to be exported to the higher dimensional case (see \cite{BL2,BHL}). Therefore, it became important to make clear similitudes and differences in the thermodynamic formalism between physicist and  (1-d) ergodic viewpoints. 

Our first result (see Theorem \ref{thmCW}) states a kind of dictionary between thermodynamic formalism in Statistical Mechanics and Probability on the one hand, and  Ergodic Theory on the other hand. More precisely we explain  with the ergodic vocabulary the first-order phase transition arising for the Curie-Weiss Model (mean field case) and make precise the link between Gibbs measures within the Physicist/Probabilistic viewpoints and the Ergodic viewpoint. 
We initially decided to focus on the mean field case for the following reasons. First, there is a large literature dealing with this topic. Second, the mean field model is naturally represented into $\{-1,+1\}^{\N}$ and exhibits ``physical phase transitions''  that we wanted to compare with ``1-d ergodic phase transitions'' in $\{-1,+1\}^{\N}$.\\
From there, a natural question was to get a similar dictionary  for the Curie-Weiss-Potts model which is a generalization of the Curie-Weiss model. This is done in Theorem \ref{thmCWP}. 

These two results are then the motivation for our main result (see Theorem \ref{theo-originalsin}). The key point is that the Hamiltonian for the Curie-Weiss model is almost equal to the square of a  Birkhoff sum. Now, Birkhoff sum is a key object in Dynamical Systems.
We thus introduce within the Ergodic viewpoint the notion of \emph{quadratic free energy}. It is equal to the entropy plus the square of an integral. We are naturally led to study a variational principle, that is which invariant measures do maximize the quadratic free energy. This maximum defines the \emph{quadratic pressure}.
At the same time, we introduce a generalized Hamiltonian in the Curie Weiss model  and show the link between the associated Gibbs measures (within Physicist/Probabilistic viewpoint) and the Gibbs measures within the Ergodic viewpoint.
We show how first order phase transitions for this generalized Curie Weiss model are related to a bifurcation into the set of measures which maximize the quadratic free energy. 
Theorem  \ref{thmCW} is thus a particular case of Theorem \ref{theo-originalsin}.

We believe that this quadratic pressure generates further possible research questions in Ergodic Theory. Some of them  are discussed later (see Subsubsection \ref{subsubsec-discus}). Similarly, we believe that our generalized Curie Weiss model may have physical interest.


Finally, we point out that Theorem \ref{theo-originalsin} is not an extension of Theorem \ref{thmCWP}. There is no obstruction to define and study the quadratic pressure for more general subshift of finite type. Nevertheless, the Hamiltonian for the Curie-Weiss-Potts model does not write itself as a square of a Birkhoff sum, because one considers a vector-valued ``potential''. This is work in progress to give an extension of Theorem \ref{thmCWP} with the flavour of Theorem \ref{theo-originalsin}.

\subsection{Precise settings and results}
\subsubsection{Ergodic and Dynamical setting}
We consider a finite set $\L$ with cardinality bigger or equal to 2. It is called the alphabet. Then we consider the one-sided full shift $\S=\L^{\N}$ over $\L$. A point $x$ in $\S$ is a sequence $x_{0},x_{1},\ldots$ (also called an infinite word) where the $x_{i}$ are in $\L$. Most of the times we shall use the notation $x=x_{0}x_{1}x_{2}\ldots$. A $x_{i}\in\L$ can either be called a letter, or a digit or a symbol. 

The shift map $\s$ is defined by 
$$\s(x_{0}x_{1}x_{2}\ldots)=x_{1}x_{2}\ldots.$$

The distance between two points $x=x_{0}x_{1}\ldots$ and $y=y_{0}y_{1}\ldots$ is given by 
$$d(x,y)=\frac1{2^{\min\{n,\ x_{n}\neq y_{n}\}}}\cdot$$
A finite string of symbols $x_{0}\ldots x_{n-1}$ is also called a \emph{word}, of length $n$. For a word $w$, its length is $|w|$. A \emph{cylinder} (of length $n$) is denoted by $[x_{0}\ldots x_{n-1}]$. It is the set of points $y$ such that $y_{i}=x_{i}$ for $i=0,\ldots n-1$. We shall also talk about $n$-cylinder instead of cylinder of length $n$. 

If $w$ is the word of finite length $w_{0}\ldots w_{n-1}$ and $x$  is a word, the concatenation $wx$ is the new word $w_{0}w_{1}\ldots w_{n-1}x_{0}x_{1}\ldots$. 

For $\psi:\S\to \R$ continuous and $\beta>0$, the \emph{pressure function} is defined by 
\begin{equation}
\label{eq-def-pressure}
\cal{P}(\beta\psi):=\sup_{\mu}\left\{h_{\mu}+\beta\int_{\S}\psi\,d\mu\right\},
\end{equation}
where the supremum is taken among the set $\cal{M}_{\s}(\Sigma)$ of $\s$-invariant probabilities on $\S$ and $h_{\mu}$ is the Kolmogorov-Sinaï entropy of $\mu$. The real parameter $\beta$ is assumed to be positive because it represents the inverse of the temperature in statistical mechanics. 
It is known that the supremum is actually a maximum  and any measure for which the maximum is attained in \eqref{eq-def-pressure} is called an \emph{equilibrium state for $\beta\psi$}. We refer the reader to \cite{bowen,ruelle} for basic notions on thermodynamic formalism in ergodic theory. 

If $\psi$ is Lipschitz continuous then the Ruelle theorem (see \cite{Ruellecmp9}) states that  for every $\beta$, there is a unique equilibrium state for $\beta\psi$, which is denoted by $\wt{\mu}_{\beta\psi}$. 
It is \emph{ergodic} and it shall be called the \emph{dynamical Gibbs measure} (DGM for short\footnote{We prefer the adjective ``dynamical'' instead of ``ergodic'' to avoid the discussion if an ergodic Gibbs measure is ergodic or not.}). It is the unique $\s$-invariant probability measure which satisfies the property that for every  $x=x_{0}x_{1}\ldots$ and for every $n$,
\begin{equation}
\label{eq-def-gibbs-ergo}
e^{-C_{\beta}}\le\frac{\wt{\mu}_{\beta\psi}([x_{0}\ldots x_{n-1}])}{e^{\be.S_{n}(\psi)(x)-n\cal{P}(\beta\psi)}}\le e^{C_{\beta}},
\end{equation}
where $C_{\beta}$ is a positive real number depending only on $\beta$ and $\psi$ (but not on $x$ or $n$), and $S_{n}(\psi)$ stands for $\psi+\psi\circ\s+\ldots+\psi\circ\s^{n-1}$.

In this setting, the {\it $\be\psi$-conformal measure} is the unique probability 
measure such that for every $x$ and for every $n$,
\begin{equation}
\label{eq-def-mesconform}
\nu_{\beta\psi}([x_{0}\ldots x_{n-1}])=\int e^{\beta S_{n}(\psi)(x_{0}\ldots x_{n-1}y)-n\cal{P}(\beta\psi)}\,d\nu_{\be\psi}(y).
\end{equation}
A precise (and more technical) definition of conformal measure is given in page \pageref{rem-meas-confor}, where the connection between conformal measures and DGM
is stated. We emphasize that in our setting, conformal measures and DGM are equivalent measures and one can obtain one from the other. 

If the choice of $\psi$ is clear we shall drop the $\psi$ and write  $\wt{\mu}_{\beta}$, $\nu_{\be}$ and $\CP(\be)$. 

\subsubsection{The Curie-Weiss model}
We consider the case $\L=\{-1,+1\}$; $\S$ will be denoted by $\S_{2}$. 

If $\omega_{0}\ldots \omega_{n-1}$ is a finite word, we set 
\begin{equation}
\label{eq-def-hamil}
H_{n}(\om):=-\frac1{2n}\sum_{i,j=0}^{n-1}\om_{j}\om_{i}.
\end{equation}
It is called the \emph{Curie-Weiss Hamiltonian}. 
The \emph{empirical magnetization} for $\om$ is $m_{n}(\om):=\disp \frac1n\sum_{j=0}^{n-1}\om_{j}$. Then we have 
\begin{equation}\label{eq-hamilbirkhoff}
H_{n}(\om)=-\frac{n}2(m_{n}(\om))^{2}.
\end{equation}

We denote by $\bb P: =\rho^{\otimes\bb N}$ the product measure on $\S_2$, where $\rho$ is the uniform measure on $\{-1,1\}$, i.e. $\rho(\{1\})=\rho(\{-1\})=\frac{1}{2}$, and we define the
\emph{probabilistic Gibbs measure} (PGM for short) $\mu_{n,\beta}$ on $\S_2$ by
\begin{equation}\label{eq-def-gibbs}
\mu_{n,\beta}(d\omega):=\frac{e^{-\beta H_{n}(\om)}}{Z_{n,\beta}}\bb P(d\omega),
\end{equation}
where $Z_{n,\beta}$ is the normalization factor
$$Z_{n,\beta}=\disp\frac{1}{2^n}\sum_{\om',\ |\om'|=n}e^{-\beta H_{n}(\om')}.$$
Note that $\munbe$ can also be viewed as a probability defined on $\L^{n}$. 

The measure $\P$ is a Bernoulli measure and is $\s$-invariant. In Ergodic Theory it is usually called the Parry-measure (see \cite{parry-pollicott}) and turns out to be  the unique measure with maximal entropy.
With our previous notations it corresponds to the DGM $\wt\mu_{0}$.

If $P_n$, $P$ are probability measures on the Borel sets of a metric space $S$, we say that $P_n$ converges weakly to $P$ if $\int_S f\,dP_n\rightarrow \int_S f\, dP$ for each $f$ in the class $C_b(S)$ of bounded, continuous real functions $f$ on $S$. In this case we write $P_{n}\stackrel[n\to+\8]{w}{\longrightarrow}P$.

Our first result concerns the weak convergence of the measures $\mu_{n,\be}$.

\begin{theorem}\label{thmCW}\emph{\textbf{Weak convergence for the CW model}}

Let $\xi_{\beta}$ be the unique point in $[0,1]$ which realizes the maximum for $$\varphi_{I}(x):=\log(\cosh(\beta x))-\frac\beta2x^{2}.$$ 
Let $\wt{\mu}_{b}$ be the dynamical Gibbs measure for $b(\BBone_{[+1]}-\BBone_{[-1]})$.
Then
\begin{equation}\label{limCW}
\mu_{n,\beta}\stackrel[n\to+\8]{w}{\longrightarrow}
\left\{\begin{aligned}
&\wt\mu_{0}& \textrm{ if } \be\leq 1,\\
&\frac{1}{2}\left[\wt{\mu}_{\beta \xi_{\beta}}+\wt{\mu}_{-\beta \xi_{\beta}}\right]&
\textrm{ if } \be > 1.
\end{aligned}\right.
\end{equation}
\end{theorem}

\begin{remark}
Actually $\munbe$ converges towards $\frac{1}{2}\left[\wt{\mu}_{\beta \xi_{\beta}}+\wt{\mu}_{-\beta \xi_{\beta}}\right]$ for every $\beta >0$ since 
we shall see that for $\be\leq 1$ we have $\xi_\be=0$, and it is clear that $\wt{\mu}_{0}=\rho^{\otimes\bb N}$.
\end{remark}

We refer to \cite{Ellis-livre}, sections IV.4 and V.9, for a discussion of the Curie-Weiss model and historical references (see also \cite{Rassoul-Seppala}, section 3.4).
We also mention that the weak convergence of $\mu_{n,\beta}$ was already proved by Orey (\cite{Orey1988}, Corollary 1.2) 
by a nice simple probabilistic argument. We remind that our motivation is the dictionary aspect and not the convergence. 

We emphasize the equality
\begin{equation}\label{eq-mnbirkhoff}
m_{n}(\om):=\disp\frac1nS_{n}(\BBone_{[+1]}-\BBone_{[-1]})(\omega)
\end{equation}
which shows that $m_{n}$ can be written as a Birkhoff mean of a continuous function. 

A consequence of \eqref{eq-mnbirkhoff} is that 
 \eqref{eq-hamilbirkhoff} can be rewritten under the form 
$$H_{n}(\om)=-\frac{n}2\left(\frac1nS_{n}(\psi)(\om)\right)^{2},$$
where $\psi:=\BBone_{[+1]}-\BBone_{[-1]}$. 


\subsubsection{The generalized Curie Weiss model}
If $\psi$ is a H\"older continuous function on $\S_2$, we define the 
\emph{generalized Curie-Weiss Hamiltonian} $H^\psi_{n}$ associated to $\psi$ by setting
$$H^\psi_{n}(\om)=-\frac{n}{2}\left(\frac1nS_{n}(\psi)(\om)\right)^{2}.$$
Then $\munbe^\psi$ is the PGM defined by 
\begin{equation}\label{eq-def-gibbs-general}
d\mu^\psi_{n,\beta}(d\omega):=\frac{e^{-\beta H^\psi_{n}(\om)}}{Z^\psi_{n,\beta}}d\bb P(\omega),
\textrm{ with } Z^\psi_{n,\beta}=\int_{\S_2}e^{-\beta H^\psi_{n}}\,d\bb P.
\end{equation}

If $\mu$ is an invariant measure on $\S_{2}$, we define its \emph{quadratic free energy}  by
$$h_{\mu}+\frac{\beta}{2}\left(\int_{\S_{2}}\psi\,d\mu\right)^{2}.$$
Then we define the \emph{quadratic pressure function (associated to $\Psi$)} by  
\begin{equation}
\label{eq-def-quadratic-pressure}
\cal{P}_2(\beta \psi):=\sup_{\mu}\left\{h_{\mu}+\frac{\beta}{2}\left(\int_{\S_2}\psi\,d\mu\right)^2\right\}.
\end{equation}
Upper semi-continuity for the entropy immediately shows that the supremum is a maximum. The function $\beta\mapsto \cal P_{2}(\beta \psi)$ is obviously convex (thus continuous).

\begin{theorem}\label{theo-originalsin}{\bf Weak convergence for the generalized Curie Weiss model}

Let $\psi$ be a H\"older continuous function on $\S_2$, let $\be$ be a positive real number.
\begin{enumerate}
\item There are finitely many invariant probabilities $m_1,\cdots,m_J$ (with $J=J(\be)$) whose quadratic free energy (for $\beta$) is maximal and thus equal to the quadratic pressure $\CP_{2}(\be\psi)$.
\item Each $m_i$ is the unique equilibrium state $\wt{\mu}_{\beta t_i\psi}$ for the potential $\be t_i\psi$.
\item The numbers $t_1,\cdots,t_J$ are the maxima of the auxiliary function 
$$\varphi_{\textrm{OS}}(t):=\cal P(\be t\psi)-\frac{\be}{2}t^2.$$
\item As $n$ goes to $+\8$, $\mu^\psi_{n,\be}$ converges weakly to a convex combination 
of the conformal measures $\nu_{\be t_{j}}$'s associated to
$\be t_{j}\psi$:
$$\munbe\stackrel[n\to+\8]{w}{\longrightarrow}\sum_{j=1}^J c_{j}\nu_{\be t_{j}}.$$
The $c_{j}$'s are well identified (see formulas \eqref{equ1-cj} and \eqref{equ2-cj}). 
\end{enumerate}
\end{theorem}

We emphasize that 
Theorem \ref{thmCW} is a particular case of Theorem \ref{theo-originalsin} with $\psi=\BBone_{[+]}-\BBone_{[-]}$. In that case the pressure is easy to compute and is equal to
$$\cal P(\be\psi)=\log 2+\log(\cosh \be),$$
and we get $\varphi_{OS}(x)=\log 2+\varphi_{I}(x)$. Note that for this particular case, the DGM is also the conformal measure. 

\subsubsection{Comparison of definitions of phase transition}
Nowadays, a phase transition in ergodic theory means the lack of analyticity for the pressure function (see \eg \cite{letelier,sarig, smirnov}). It is known that this notion is transversal to the number of equilibrium states. One can have a loss of analyticity with only one equilibrium state (see  the Manneville-Pomeau example with good parameters, \cite{thaler}) or analyticity with several equilibrium states (see \cite{leplaideur-butterfly}).

For the quadratic pressure, things may be different. We remind that $z\mapsto \CP(z\psi)$ is analytic (for H\"older continuous $\psi$). Each $t_{i}$ is a maximum for $\varphi_{\textrm{OS}}$ and then satisfies $\disp\CP'(\be t_{i})=t_{i}$. It is thus highly probable that $t_{i}(\be)$ is locally analytic (and surely locally $\CC^{\8}$). 
Then, the quadratic pressure satisfies 
$$\CP_{2}(\be)=h_{\wt\mu_{\be t_{i}}\psi}+\frac\be2\left(\int_{\S_{2}}\psi\,d\wt\mu_{\be t_{i}\psi}\right)^{2}=\CP(\be t_{i}\psi)-\be t_{i}+\frac\be2t_{i}^{2}.$$
It is thus reasonable to expect $\CP_{2}(\be)$ to be at least piecewise $\CC^{\8}$ and even probably piecewise $\CC^{\omega}$. Moreover, we expect the borders of intervals of analyticity to be exactly where there is a change in the number of $t_{i}$'s.  \\
It is therefore very likely that  the loss of analyticity for the quadratic pressure is equivalent to a bifurcation in the number of ``quadratic'' equilibrium states. Actually, this is corroborated by Theorem \ref{thmCW}, where the quadratic pressure is piecewise analytic (and not analytic) and the number of quadratic equilibrium states change with respect to $\be$ exactly where analyticity fails. 

\subsubsection{Some consequences of Theorem \ref{theo-originalsin}}\label{subsubsec-discus}
Several questions naturally arise from Theorem \ref{theo-originalsin}. At that stage, we do not have more precise conjectures or ideas for answers. 
%
%

$\bullet$ For more geometric dynamical systems, one usually considers or studies the special class of \emph{physical or/and SRB-measures}. These measures are usually considered as the most natural ones  with the measures of maximal entropy. It is clear that measures of maximal entropy also maximize $\disp h_{\mu}+\left(\int_{\S_{2}}\psi\,d\mu\right)^{2}$ for $\psi\equiv 0$. 

A natural question is thus to  know if  for a system admitting one SRB-measure, there exists some potential $\psi$ such that the SRB measure maximizes the quadratic free energy $\disp h_{\mu}+\left(\int_{\S_{2}}\psi\,d\mu\right)^{2}$. 

\bigskip
$\bullet$ More generally, one can ask how big is the set of measures which maximize the quadratic pressure for some potential $\psi$ ? It is for instance known that any ergodic measure is an equilibrium state for some continuous potential (see \cite[Cor. 3.17]{ruelle}). Does it still hold for quadratic pressure ?

\bigskip
$\bullet$ Ergodic Optimization  studies what happens to DGM $\wt\mu_{\be\psi}$ as $\be$ goes to $+\8$. It is known that any accumulation point maximizes the integral of $\psi$ among invariant measures. The goal is to study if there is convergence and how is the limit selected among the simplex of $\psi$-\emph{maximizing measures}. 
The same kind of questions may be studied with the quadratic pressure. We point-out that non-linearity may introduce very new and different phenomena compared to the ``usual pressure''.

\subsubsection{The Curie-Weiss-Potts model. Probabilistic settings 2 and result}
The \emph{Curie-Weiss-Potts model} will be for $\L=\{\theta^{1},\ldots ,\theta^{q}\}$ with $q>2$. In that case we shall write $\S_{q}$ instead of $\S$. 

The Curie-Weiss-Potts Hamiltonian is defined for a finite word $\omega=\om_0\cdots\om_{n-1}$ by
\begin{equation}
\label{eq-def-hamilCWP}
H_{n}(\om):=-\frac1{2n}\sum_{i,j=0}^{n-1}\BBone_{\om_{j}=\om_{i}}.
\end{equation}
We define the vector
$L_{n}(\omega)=(L_{n,1}(\om),\cdots,L_{n,q}(\om))$ where 
$$L_{n,k}(\om)=\sum_{i=0}^{n-1}\BBone_{\omega_i=\theta^k}$$
is the number of digits of $\omega$ which take the value $\theta^k$,
so that we can write
$$\sum_{i,j=0}^{n-1}\BBone_{\om_{j}=\om_{i}}=
\sum_{k=1}^q\pare{\sum_{i=0}^{n-1}\ind_{\om_{i}=\theta^k}}^2
=\|L_n(\omega)\|^2,$$
where $\|\cdot\|$ stands for the Euclidean norm on $\bb R^q$.

We denote by $\bb P: =\rho^{\otimes\bb N}$ the product measure on $\S_q$, where $\rho$ is the uniform measure on $\Lambda$, i.e. $\rho=\frac{1}{q}\sum_{k=1}^q\delta_{\theta^k}$, and we define the probabilistic Gibbs measure $\mu_{n,\beta}$ on $\S_q$ by
\begin{equation}\label{eq-def-gibbsCWP}
\mu_{n,\beta}(d\omega):=\frac{e^{-\beta H_{n}(\om)}}{Z_{n,\beta}}\bb P(d\omega)
=\frac{e^{\frac{\be}{2n}\|L_{n}(\om)\|^2}}{Z_{n,\beta}}\bb P(d\omega),
\end{equation}
where $Z_{n,\beta}$ is the normalization factor
$$Z_{n,\beta}=\disp\frac{1}{q^n}\sum_{\om',\ |\om'|=n}e^{\frac{\be}{2n}\|L_{n}(\om')\|^2}.$$

Now we can state the analog of Theorem \ref{thmCW}.
\begin{theorem}\label{thmCWP}
\emph{\textbf{Weak convergence for the CWP model}}

For $1\leq k\leq q$, $b\in\bb R$, let $\wt{\mu}^{k}_{b}$ be the dynamical Gibbs measure for $b \BBone_{[\theta^k]}$.
Let $\be_c=\frac{2(q-1)\log(q-1)}{q-2}$. For $0<\be<\be_c$ set $s_\be=0$ and for $\beta\geq \beta_c$ let $s_\be$ be the largest solution of the equation
\begin{equation}\label{sbeta}
s=\frac{e^{\be s}-1}{e^{\be s}+q-1}.
\end{equation}

Then,
\begin{equation}\label{limCWP}
\mu_{n,\beta}\stackrel[n\to+\8]{w}{\longrightarrow}
\left\{\begin{aligned} 
&\rho^{\otimes\bb N}& \textrm{ if } 0<\be<\be_c,\\
&\frac{1}{q}
\sum_{k=1}^q \wt{\mu}_{\beta s_\be}^{k}& \textrm{ if } \be>\be_c ,\\
&\frac{A\, \wt{\mu}_{0}^{1}+B\,\sum_{k=1}^q \wt{\mu}_{\beta_c s_{\be_c}}^{k}}{A+qB}&
 \textrm{ if } \be=\be_c, \end{aligned}\right.
\end{equation}
with $A=\pare{1-\frac{\beta_c}{q(q-1)}}^{\frac{q-2}{2}}$ and 
$B=\pare{1-\frac{\beta_c}{q}}^{\frac{q-2}{2}}$.
\end{theorem}

\begin{remark}
Actually $\munbe$ converges towards $\frac{1}{q}
\sum_{k=1}^q \wt{\mu}_{\beta s_\be}^{k}$ for every $\beta \neq\be_c$ since 
$s_\be=0$ for $\be<\be_c$, and it is clear that $\wt{\mu}^{k}_{0}=\rho^{\otimes\bb N}$
for each $1\leq k\leq q$.
\end{remark}

We refer to \cite{EllisWang90} for a discussion of the Curie-Weiss-Potts model and historical references.
Orey (\cite{Orey1988}, Theorem 4.4) mentions the weak convergence 
of $\mu_{n,\beta}$ towards an explicit atomic measure, but he makes a mistake concerning the case $\beta=\beta_c$, as pointed out in \cite{EllisWang90}.

\subsection{Plan of the paper}
The paper is composed as follows. 

In Section \ref{sec-prooforiginalsin} we prove Theorem \ref{theo-originalsin}, in Section \ref{sec-proofthcwp} we prove Theorem \ref{thmCWP}. Both proofs are based on the convergence of $\munbe(C)$ where $C$ is a cylinder in $\S$. 

We point out that in Theorem \ref{theo-originalsin} the proofs of the parts (3)-(4) and of the parts (1)-(2) are independent.

Theorem \ref{thmCW} is a simple consequence of Theorem \ref{theo-originalsin} as said above. 


\section{Proof of Theorem \ref{theo-originalsin}}\label{sec-prooforiginalsin}
 

\subsection{Convergence of \texorpdfstring{$\mu^\psi_{n,\be}$}{}}
To lighten the notations we drop the $\psi$ in $H^\psi_{n}$,
$\munbe^\psi$, $Z^\psi_{n,\beta}$.
To prove the weak convergence of $\munbe$ towards a measure $\mu$, it is enough to show that for every cylinder $C$,
\begin{equation}
\lim\limits_{n\to\8} \munbe(C)=\mu(C).
\end{equation}
Let $\omega=\om_{0}\ldots \om_{p-1}$ be a finite word of length $p$, let  
$n>p$. 
We use the equality 
$$e^{a^2}=\frac{1}{\sqrt{2\pi}}\int_{-\8}^{+\8}e^{-\frac{x^2}{2}+\sqrt 2 a x}\,dx,$$
sometimes called the Hubbard-Stratonovich transformation (\cite{Hubbard},\cite{Stratonovic}), to compute the following.

\begin{eqnarray}
Z_{n,\be}\munbe([\om])&=&\int_{\S_2} e^{\frac{\be}{2n}(S_{n}(\psi)(\al))^{2}}\BBone_{[\om]}(\al)d\bb{P}(\al)\nonumber\\
&=&\int_{\S_2}\frac{1}{\sqrt{2\pi}}\int_{-\8}^{+\8} e^{-\frac{x^2}{2}}
e^{\sqrt{\frac{\beta}{n}}xS_n(\psi)(\al)}\BBone_{[\om]}(\al)\,dx\,d\bb{P}(\al),\nonumber\\
&=& \sqrt\frac{\be n}{2\pi}\int_{-\8}^{+\8} e^{-n\frac{\be}{2}z^{2}}\int_{\S_2}
e^{\be z S_n(\psi)(\al)}\BBone_{[\om]}(\al)d\bb P(\al)\,dz,\label{eq1-cvmunbe}
\end{eqnarray}
where we have made the change of variable $\be z=\disp\sqrt\frac\be{n}x$. 

Let us define the Transfer operator $\CL_{\xi}$, depending on a real or complex parameter $\xi$, by
$$\CL_{\xi}(T)(x):=\sum_{y,\ \s(y)=x}e^{\xi\psi(y)}T(y).$$
Then for every $n\in\bb N$,
\begin{equation}\label{L-puissance-n}
\CL^n_{\xi}(T)(x)=\sum_{y,\ \s^n(y)=x}e^{\xi S_n\psi(y)}T(y).
\end{equation}
We remind the following properties for $\CL_{\xi}$ (see \cite{bowen}, \cite{parry-pollicott}): 
\begin{prop}
\label{rem-meas-confor}
The operator $\CL_{\xi}$ acts on continuous and H\"older continuous functions. On H\"older continuous functions, its spectral radius $\l_{\xi}$ is a simple dominating eigenvalue. We denote by $H_{\xi}$ a positive associated eigenfunction. The rest of the spectrum of $\CL_{\xi}$ is included into the disk of radius $\l_{\xi}e^{-\eps(\xi)}$ for some positive $\eps(\xi)$. 
The dual operator (for continuous functions) acts on measures. There exists a unique probability measure $\nu_{\xi}$ which is the eigen-measure for the eigen-value $\l_{\xi}$. The measure $\nu_{\xi}$ is the conformal measure. The DGM is then equal to
$$d\wt\mu_{\xi}=H_{\xi}d\nu_{\xi},$$
where $H_{\xi}$ is normalized such that $\wt\mu_\xi$ is a probability measure.  The pressure is $\log\l_{\xi}$. 
\end{prop}
Therefore we have 
\begin{equation}
\label{equ-clnom}
\CL^{n}_{\xi}(\BBone_{[\om]})(x)=\l_{\xi}^{n}\nu_{\xi}([\om])H_{\xi}(x)+\l_{\xi}^{n}e^{-n\eps(\xi)}T(n,\xi)(x)
\end{equation}
with $||T(n,\xi)||_{\8}\le 1$. 

For $\al\in {\S_2}$ one writes $\al=\bar\al\theta$ with $\bar\al$ equal to the suffix of length $n$ of $\al$ and $\theta$ in ${\S_2}$. 
Using \eqref{L-puissance-n} we can rewrite \eqref{eq1-cvmunbe} as
\begin{eqnarray}
Z_{n,\be}\munbe([\om])&=& \frac1{2^{n}}\sqrt\frac{\be n}{2\pi}\int_{-\8}^{+\8} e^{-n\frac{\be}{2}z^{2}}\int_{\S_2}
\sum_{\bar\al} e^{\be z S_n(\psi)(\bar\al\theta)}\BBone_{[\om]}(\bar\al)d\bb P(\theta)\,dz,\nonumber\\
&=& \frac1{2^{n}}\sqrt\frac{\be n}{2\pi}\int_{-\8}^{+\8} e^{-n\frac{\be}{2}z^{2}}\int_{\S_2} \CL^{n}_{\be z}(\uncom)(\theta)d\bb P(\theta)\,dz.\label{eq2-cvmunbe}
\end{eqnarray}
The normalization factor $Z_{n,\be}$ can be computed by replacing $[\om]$ by ${\S_2}$, and we get
\begin{equation}
\label{quotient}
\munbe([\om])=\dfrac{\int_{-\8}^{+\8} e^{-n\frac{\be}{2}z^{2}}\int_{\S_2} \CL^{n}_{\be z}(\uncom)(\theta)d\bb P(\theta)\,dz}
{\int_{-\8}^{+\8} e^{-n\frac{\be}{2}z^{2}}\int_{\S_2} \CL^{n}_{\be z}(\1)(\theta)d\bb P(\theta)\,dz}=:\dfrac{N_{n,\be}}{D_{n,\be}}.
\end{equation}
Using \eqref{equ-clnom} we get
\begin{multline}\label{eq3-cvmunbe}
N_{n,\be}= 
\int_{-\8}^{+\8} e^{-n\frac{\be}{2}z^{2}+n\log\l_{\be z}}\\
\left[\int_{\S_2} \nu_{\be z}([\om])H_{\be z}(\theta)+e^{-n\eps(\be z)}T(n,\be z)(\theta)d\bb P(\theta)\right]\,dz.
\end{multline}
We want to use the Laplace method but the last term in the inner integral depends on $n$. This term converges to zero as $n$ goes to infinity but the speed of convergence depends on $z$ and $|z|$ may go to infinity.
Setting $A:=\|\psi\|_\infty$, we deduce from \eqref{L-puissance-n} that for every $n$, every $\xi$ and every $T$ continuous
\begin{equation}
\label{eq4-cvmunbe}
||\CL_{\xi}^{n}(T)||_{\8}\le 2^{n}e^{n\xi A}||T||_{\8}. 
\end{equation}
Therefore the term in the integral defining the numerator $N_{n,\be}$
in \eqref{quotient} is bounded from above by 
$\disp e^{-n\frac\be2z^{2}+n\log 2+n\be zA}$. Furthermore, there exists $Z(\be)$ such that for $|z|>Z(\be)$
\begin{equation}\label{maj}
-\frac\be2z^{2}+\log 2+\be zA\le-\frac\be4z^{2}
\end{equation}
holds, from which we deduce that there exists $\kappa(\be)>0$ such that for every $n> p$,
\begin{equation}\label{compact}
\int_{|z|\geq \kappa(\be)}e^{-n\frac{\be}{2}z^{2}}
\int_{\S_2} \CL^{n}_{\be z}(\uncom)(\theta)d\bb P(\theta)\,dz
\leq e^{-n\kappa(\be)}.
\end{equation}

From this we claim that the computation of the integral in \eqref{eq3-cvmunbe} can be done for $z$ in the compact set $[-Z(\be),Z(\be)]$ instead of $\bb R$. As the spectral gap $\xi\mapsto\eps(\xi)$ is lower semi-continuous (see \cite{hennion-herve}), the map $z\mapsto \eps(\be z)$ attains its infimum on $[-Z(\be),Z(\be)]$
so that $\int_{\S_2} e^{-n\eps(\be z)}T(n,\be z)(\theta)d\bb P(\theta)$
converges uniformly to zero on $[-Z(\be),Z(\be)]$.
This yields that one can use the Laplace method for the convergence in   \eqref{eq3-cvmunbe}, as we now explain.

The Laplace method shows that if $\varphi:I\to\bb R$
is a twice continuously differentiable function, if $\varphi'$ vanishes on a single point $\xi$ in the interior of the interval $I$, with $\varphi''(\xi)<0$, and if $f:I\to \bb R$ is continuous with $f(\xi)\neq 0$, then 
\begin{equation}
\label{eq-laplacemethod}
\int_I e^{n\varphi(y)} f(y) dy\,\underset{n\to\infty}{\sim}\,
\frac{\sqrt{2\pi}}{\sqrt{|\varphi''(\xi)|}}e^{n\varphi(\xi)}f(\xi)n^{-1/2},
\end{equation}
where by $u_{n}\underset{n\to\infty}{\sim}v_{n}$ we mean that $u_{n}=v_{n}(1+\epsilon(n))$ with $\lim_{\ninf}\epsilon(n)=0$. 
We refer to \cite{Erdelyi} for a report about the Laplace method, and for a generalization to the case where the least integer $k$ such that 
$\varphi^{(k)}(\xi)\neq 0$ is greater than two.
Of course when $\varphi$ has a finite number of maxima we may break up the integral into a finite number of integrals so that in each integral $\varphi$ reaches its maximum at only one interior point.

In our case we claim that the function $\disp \varphi_{\textrm{OS}}:z\mapsto -\frac\be2z^{2}+\log\l_{\be z}$ admits only finitely many maxima. Indeed, for $z\notin [-Z(\be),Z(\be)]$, $\varphi(z)<-\frac\be4z^{2}<0$ (this is a consequence of \eqref{eq4-cvmunbe} and \eqref{maj}) and $\varphi_{\textrm{OS}}(0)=\log2>0$. Therefore, the maxima for $\varphi_{\textrm{OS}}$ must be in the compact interval $[-Z(\be),Z(\be)]$ . If there are infinitely many, there must be some accumulation point. As $\varphi_{\textrm{OS}}$ is analytic in some complex neighborhood of $[-Z(\be),Z(\be)]$, it must be equal to the constant function, which is clearly not the case. 
Let $t_1,\cdots,t_J$ the points where $\varphi_{\textrm{OS}}$ attains its maximum.
We write the integral \eqref{eq3-cvmunbe} over the segment $[-Z(\be),Z(\be)]$ as
a finite sum of integrals over segments $[a_j,b_j]$ where each segment $[a_j,b_j]$ contains exactly one of the points $t_j$, $1\leq j\leq J$.

We state the following lemma, which is an immediate adaptation of the Laplace method. 
\begin{lemma}\label{Laplace-uniforme}
Let $\varphi: [a,b]\to\bb R$ a function of class $C^2$, with $\varphi'$ vanishing on a single point $c$ in $]a,b[$ and $\varphi''(c)<0$. Let $(f_n)_{n\geq 1}$, $f$
some continuous functions from $[a,b]$ to $\bb R$ such that $f_n$ converges to $f$ uniformly on $[a,b]$, and $f(c)\neq 0$. 
Then as $n\to\infty$
$$\int_a^b e^{n\varphi(x)}f_n(x)\,dx\sim \sqrt{\frac{\pi}{2|\varphi''(c)|}}e^{n\varphi(c)}f(c)n^{-1/2}.$$
\end{lemma}

We apply this lemma on every $[a_j,b_j]$ to the functions $f_n$ defined by
$$f_n(z)=\int_{\S_2} \nu_{\be z}([\om])H_{\be z}(\theta)+e^{-n\eps(\be z)}T(n,\be z)(\theta)d\bb P(\theta).$$
The functions $f_{n}$ converge uniformly on $[a_j,b_j]$ to $f$ defined by
$$f(z)=\left(\int_{\S_2} H_{\be z}(\theta)d\bb P(\theta)\right)\nu_{\be z}([\om]).$$
Putting together \eqref{compact} and the result of Lemma \ref{Laplace-uniforme}
applied to every $[a_j,b_j]$, assuming for the moment that 
$\varphi''_{\textrm{OS}}(t_j)<0$ for every $j=1,\cdots,J$, we obtain that 
$N_{n,\be}$ is equivalent when $n$ goes to infinity to 
$$\sqrt{\frac{\pi}{2n}}e^{n\varphi_{\textrm{OS}}(t_1)}
\sum_{j=1}^J\dfrac{\disp\left(\int_{\S_2} H_{\be t_j}(\theta)d\bb P(\theta)\right)\nu_{\be t_j}([\om])}{\sqrt{|\varphi''_{\textrm{OS}}(t_j)|}}$$
and $D_{n,\be}$ is equivalent to 
$$\sqrt{\frac{\pi}{2n}}e^{n\varphi_{\textrm{OS}}(t_1)}
\sum_{j=1}^J\dfrac{\left(\disp\int_{\S_2} H_{\be t_j}(\theta)d\bb P(\theta)\right)}{\sqrt{|\varphi''_{\textrm{OS}}(t_j)|}}.$$
Recalling \eqref{quotient} we get that
$\munbe([\om])$ converges to 
\begin{equation}\label{combi}
\sum_{j=1}^J c_{j}\nu_{\be t_{j}}([\om]),
\end{equation}
where
\begin{equation}
\label{equ1-cj}
c_{j}:=\dfrac{\disp
\frac{\disp\int_{\S_2} H_{\be t_{j}}d\bb P}{\sqrt{|\varphi''_{\textrm{OS}}(t_j)|}}}
{\disp\sum_{i=1}^J \frac{\disp\int_{\S_2} H_{\be t_{i}}d\bb P}
{\sqrt{|\varphi''_{\textrm{OS}}(t_i)|}}}.
\end{equation}

If $\varphi''_{\textrm{OS}}(t_i)=0$ then the contribution of the integral over $[a_i,b_i]$ is of order $e^{n\varphi_{\textrm{OS}}(t_{i})}n^{-1/{k_{i}}}$ where
$k_{i}$ is the least integer such that $\varphi^{(k_{i})}_{\textrm{OS}}(t_i)<0$. Note that all $\varphi_{\textrm{OS}}(t_{i})$ are equal and the $k_{i}$'s are all even numbers because $\varphi_{\textrm{OS}}$ reaches its maximum at each $t_{i}$.

Let $K:=\max{k_{i}}$ and let $\CI$ be the set of indexes $i$'s such that $k_{i}=K$. 
Then
we still get the convergence of $\munbe([\om])$
to a convex combination \eqref{combi} of measures $\nu_{\be t_j}$'s,
but with $c_j=0$ whenever $j\notin \CI$ and 
\begin{equation}
\label{equ2-cj}
c_{j}:=\dfrac{\disp
\frac{\disp\int_{\S_2} H_{\be t_{j}}d\bb P}{|\varphi^{(K)}_{\textrm{OS}}(t_j)|^{1/K}}}
{\disp\sum_{i\in\CI}^J \frac{\disp\int_{\S_2} H_{\be t_{i}}d\bb P}
{|\varphi^{(K)}_{\textrm{OS}}(t_i)|^{1/K}}}
\end{equation}
for $j\in\CI$. This finishes the proof of part (4) of Theorem \ref{theo-originalsin}. 


\subsection{Measures maximizing the quadratic pressure}
We want to determine the invariant measures $m$ which maximize 
$$h_{m}+\frac\be2\left(\int_{\S_2}\psi\,d m\right)^{2}.$$
We set $\ol A$ (resp. $\ul A$) for $\disp \max_{m\ \s-inv}\int_{\S_2} \psi\,d m$ (resp. $\disp \min_{m\ \s-inv}\int_{\S_2} \psi\,d m$). 
For $z\in\bb R$, we set 
$$\ol H(z):=\left\{\begin{aligned}
&\max_{m\ \s-inv}\left\{h_{m},\,\int_{\S_2}\psi\,dm=z\right\}\textrm{ if }z\in [\ul A,\ol A],\\
&-\infty \textrm{ if not }.\end{aligned}\right.$$
We point out the equality 
$$\CP_{2}(\be\psi):=\max_{m}\left\{h_{m}+\frac\be2\left(\int_{\S_2}\psi\,dm\right)^{2}\right\}=\max_{z\in[\ul A,\ol A]}\left\{\ol H(z)+\frac\be2z^{2}\right\}.$$

Let us set \label{eq-deffibar}$\disp\fibar(z):=\ol H(z)+\frac\be2z^{2}$. 
We claim that the maxima of $\varphi_{\textrm{OS}}$ and $\fibar$ are the same.
First we observe that
$$\begin{aligned}
\CP(t\psi):=\max_{m}\left\{h_{m}+t \int_{\S_2}\psi\,dm\right\}
&=\max_{z\in\bb R}\left\{\ol H(z)+t z\right\}\\
&=\max_{z\in\bb R}\left\{tz-(-\ol H(z))\right\}.\end{aligned}$$
As the function $-\ol H$ is convex lower semi-continuous, we deduce from the duality property of the Fenchel-Legendre transform (see for instance \cite{DemboZeitouni}, Lemma 4.5.8) that 
\begin{equation}\label{prop-minmax}
\ol H(z)=\inf_{t\in \R}\left\{\CP(t\psi)-tz\right\}.
\end{equation}
\begin{lemma}
\label{lem-fibarlefi}
For every $z$ in $[\ul A,\ol A]$, 
$\disp \fibar(z)\le \varphi_{\textrm{OS}}(z)$.
\end{lemma}
\begin{proof}
Let $t=\be z$. Using \eqref{prop-minmax} we get 
$$\fibar(z)\le \CP(t\psi)-tz+\frac\be2z^{2}=\CP(\be z\psi)-\frac\be2z^{2}=\varphi_{\textrm{OS}}(z).$$
\end{proof}

\begin{lemma}
\label{lem-maxfifibar}
$\fibar(z)$ is maximal if and only if $\varphi_{\textrm{OS}}(z)$ is maximal. In that case, $\fibar(z)=\varphi_{\textrm{OS}}(z)$. 
\end{lemma}
\begin{proof}
Let $z$ be a maximum for $\varphi_{\textrm{OS}}$. Then, it is a critical point for $\varphi_{\textrm{OS}}$. This yields
$$\be\CP'(\be z)=\be z.$$
In other words, $\disp \int_{\S_2}\psi\,d\wt\mu_{\be z}=z$ because $\disp \CP'(t\psi)=\int_{\S_2}\psi\,d\wt\mu_{t}$ holds for every $t$. 
Then,
\begin{eqnarray*}
\fibar(z)\ge h_{\wt\mu_{\be z}}+\frac\be2z^{2}&=&h_{\wt\mu_{\be z}}+\be z^{2}-\frac\be2z^{2}\\
&=&h_{\wt\mu_{\be z}}+\be z\int_{\S_2}\psi\,d\wt\mu_{\be z}-\frac\be2z^{2}\\
&=& \CP(\be z\psi)-\frac\be2z^{2}=\varphi_{\textrm{OS}}(z)\ge \fibar(z). 
\end{eqnarray*}
This means that $\fibar(z)=\varphi_{\textrm{OS}}(z)$ holds. On the other hand  for any $z'$,
$$\fibar(z')\le \varphi_{\textrm{OS}}(z')\le \varphi_{\textrm{OS}}(z)=\fibar(z),$$
which shows that $z$ is also a maximum for $\fibar$. 

Conversely, if $z$ is a maximum for $\fibar$, let $z'$ be any maximum for $\varphi_{\textrm{OS}}$. We get 
$$\fibar(z)\ge \fibar(z')=\varphi_{\textrm{OS}}(z')\ge \varphi_{\textrm{OS}}(z)\ge\fibar(z).$$
This shows that $z$ is also a maximum for $\varphi_{\textrm{OS}}$. 
\end{proof}
Now we are ready to finish the proof of Theorem \ref{theo-originalsin}.
Indeed let $m$ maximizing
$$h_{m}+\frac\be2\left(\int_{\S_2}\psi\,d m\right)^{2}.$$
Then $z:=\ds\int_{\S_2}\psi\,dm$ is a maximum for $\fibar$, hence according to Lemma
\ref{lem-maxfifibar} $z$ is a maximum for 
$\varphi_{\textrm{OS}}$ with $\varphi_{\textrm{OS}}(z)=\fibar(z)$.
Therefore there exists $i\in\interventier{1}{J}$ such that $z=t_i$,
and
$$h_m+\frac{\be}{2} t_i^2=\cal P(\be t_i \psi)-\frac{\be}{2} t_i^2.$$
We deduce that
$$h_m+\be t_i^2=\cal P(\be t_i \psi)=h_m+\be t_i \ds\int_{\S_2}\psi\,dm,$$
which implies that $m=\wt{\mu}_{\beta t_i\psi}$.
It remains to prove that each $\wt{\mu}_{\beta t_i\psi}$ does maximize
$$h_{m}+\frac\be2\left(\int_{\S_2}\psi\,d m\right)^{2}.$$
But this is immediate since
$$\cal P'(\be t_i)=\int_{\S_2}\psi\,d\wt{\mu}_{\beta t_i\psi}=t_i$$
and $t_i$ is a maximum for $\fibar$.

\section{Proof of Theorem \ref{thmCWP}} \label{sec-proofthcwp}

In a first step we use an auxiliary function $\varphi_{P}$. Note that this function was already studied by Ellis and Wang in \cite{EllisWang90}. Then we deduce that $\munbe(C)$ converge for any cylinder $C$. In a second step we identify the limit as the relevant convex combination of dynamical measures.

\subsection{Auxiliary function \texorpdfstring{$\varphi_{P}$}{} and convergence for \texorpdfstring{$\munbe$}{}}
We shall need the function $\varphi_{P}$ defined on $\bb R^{q}$ by
\begin{equation}\label{eqphi}
\varphi_P(z)=-\frac{\beta}{2}\|z\|^2+\log\sum_{k=1}^qe^{\be z_k}.
\end{equation}
This function attains its maximum on $\bb R^q$ since $\varphi_P(z)\leq -c\|z\|^2$
as $\|z\|$ tends to $\8$.
We recall Theorem 2.1 of \cite{EllisWang90}, which describes precisely the global maximum points of $\varphi_P$.

\begin{theorem}\label{EllisWang}\emph{(Ellis Wang \cite{EllisWang90})}

Let $\be_c=\frac{2(q-1)\log(q-1)}{q-2}$. For $0<\be<\be_c$ set $s_\be=0$ and for $\beta\geq \beta_c$ let $s_\be$ be the largest solution of the equation
\begin{equation}
s=\frac{e^{\be s}-1}{e^{\be s}+q-1}.
\end{equation}
The function $\be\mapsto s_\be$ is strictly increasing on the interval $[\be_c,+\8[$,  $s(\be_c)=\frac{q-2}{q-1}$, and $\lim_{\be\to\8} s_\be=1$.

Denote by $\phi$ the function from $[0,1]$ into $\bb R^q$ defined by
$$\phi(s)=\pare{\frac{1+(q-1)s}{q},\frac{1-s}{q},\cdots,\frac{1-s}{q}},$$
the last $(q-1)$ components all equal $\frac{1-s}{q}$.
Let $K_\beta$ denote the set of global maximum points of the symmetric function $\varphi_P$.
Define $\nu^0=\phi(0)=\pare{\frac{1}{q},\cdots,\frac{1}{q}}$. For $\beta\geq \beta_c$, define 
$\nu^1(\beta)=\phi(s_\be)$ and let $\nu^i(\beta)$, $i=2,\cdots, q$ denote the points in $\bb R^q$ obtained by interchanging the first and ith coordinates of $\nu^1(\beta)$. Then
$$K_\be=\left\{\begin{aligned}
&\{\nu^0\} &\textrm{ for } 0<\beta<\beta_c,\\
&\{\nu^1(\beta),\nu^2(\beta),\cdots,\nu^q(\beta)\} &\textrm{ for } \beta>\beta_c,\\
&\{\nu^0,\nu^1(\beta_c),\nu^2(\beta_c),\cdots,\nu^q(\beta_c)\} &\textrm{ for } \beta=\beta_c.
\end{aligned}\right.$$
For $\beta\geq \beta_c$ the points in $K_\beta$ are all distinct. 
\end{theorem}

We fix a finite word $\om=\om_0\cdots\om_{p-1}$ of length $p$ and we compute the limit
of $\munbe([\omega])$.

\begin{lemma}\label{lemCWP}

\begin{equation*}
\lim\limits_{n\to\8}\munbe([\om])=
\left\{\begin{aligned}&\frac{1}{q^p}\textrm{ if }\beta<\beta_c,\\
&\frac{1}{q}\frac{1}{(e^{\beta s_\be} +q-1)^{p}}\,\sum_{k=1}^q e^{\be s_\be L_{p,k}(\om)}\textrm{ if }\beta> \beta_c,\\
&\frac{\frac{A}{q^p}+\,
\frac{B}{(e^{\beta s_\be} +q-1)^{p}}\sum_{k=1}^q e^{\be_c s_{\be_c} L_{p,k}(\om)}}{A+qB}\textrm{ if }\beta= \beta_c.\end{aligned}\right.
\end{equation*}
\end{lemma}

\begin{proof}
We want to evaluate the limit of
$$\munbe([\om])=\sum_{\al,\ |\al|=n-p}\munbe([\om\al])
=\frac{\sum\limits_{\al,\ |\al|=n-p}e^{\frac{\be}{2n}\|L_{n}(\om\al)\|^2}}{\sum\limits_{\al,\ |\al|=n}e^{\frac{\be}{2n}\|L_{n}(\al)\|^2}}.$$
With the help of the identity
\begin{equation}
e^{\|u\|^2}=\frac{1}{(2\pi)^{q/2}}\int_{\bb R^q} \exp\pare{-\frac{1}{2}\|y\|^2+\sqrt 2\langle y,u\rangle}\,dy,
\end{equation}
and noticing that $L_n(\om\al)=L_p(\om)+L_{n-p}(\al)$,
we write
\begin{eqnarray*}
\sum_{\al,\ |\al|=n-p}e^{\frac{\be}{2n}\|L_{n}(\om\al)\|^2}
&=&\frac{1}{(2\pi)^{q/2}}\int_{\bb R^q} e^{-\frac{1}{2}\|y\|^2}\sum_{\al}
e^{\sqrt{\frac{\beta}{n}}\langle y,L_n(\omega\al)\rangle}\,dy\\
&=&\frac{1}{(2\pi)^{q/2}}\int_{\bb R^q} e^{-\frac{1}{2}\|y\|^2+\sqrt{\frac\beta{n}}\langle y,L_p(\om)\rangle}\sum_{\al}
e^{{\sqrt{\frac{\beta}{n}}\langle y,L_{n-p}(\al)\rangle}}\,dy.
\end{eqnarray*}
It is easily seen that
$$\sum_{\al,|\al|=n-p}
e^{{\sqrt{\frac{\beta}{n}}\langle y,L_{n-p}(\al)\rangle}}
=\pare{\sum_{k=1}^q e^{\sqrt{\frac{\beta}{n}}y_k}}^{n-p},$$
therefore we get 
\begin{multline*}\sum_{\al,\ |\al|=n-p}e^{\frac{\be}{2n}\|L_{n}(\om\al)\|^2}=\\
\frac{1}{(2\pi)^{q/2}}\int_{\bb R^q} \exp\pare{-\frac{1}{2}\|y\|^2+
\sqrt{\frac{\beta}{n}}\langle y,L_p(\om)\rangle+(n-p)\log\pare{\sum_{k=1}^q e^{\sqrt{\frac{\beta}{n}}y_k}}}\,dy.
\end{multline*}
Now we make the change of variable $\beta z=\sqrt{\frac{\beta}{n}}y$, and we obtain
\begin{equation}
\sum_{\al,\ |\al|=n-p}e^{\frac{\be}{2n}\|L_{n}(\om\al)\|^2}=
\pare{\frac{n\be}{2\pi}}^{q/2}\int_{\bb R^q} e^{n\varphi_P(z)}f(z)\,dz,
\end{equation}
where $\varphi_P$ was defined in \eqref{eqphi} and $f$ is defined on $\bb R^q$ by
\begin{equation}\label{eqf}
f(z)=\exp\pare{\beta\langle z,L_p(\om)\rangle -p\log\pare{\sum_{k=1}^qe^{\beta z_k}}}.
\end{equation}
Similarly, $p=0$ yields
$$\sum_{\al,\ |\al|=n}e^{\frac{\be}{2n}\|L_{n}(\al)\|^2}=
\pare{\frac{n\be}{2\pi}}^{q/2}\int_{\bb R^q} e^{n\varphi_P(z)}\,dz,$$
hence
$$\munbe([\om])=\frac{\sum_{\al,\ |\al|=n-p}e^{\frac{\be}{2n}\|L_{n}(\om\al)\|^2}}{\sum_{\al,\ |\al|=n}e^{\frac{\be}{2n}\|L_{n}(\al)\|^2}}
=\frac{\int_{\bb R^q} e^{n\varphi_P(z)}f(z)\,dz}{\int_{\bb R^q} e^{n\varphi_P(z)}\,dz}.$$
We denote by $D\varphi_P(z)$, respectively $\CH(z)$, the gradient, respectively the Hessian matrix, of $\varphi_P$ at $z$.
It is proved in Proposition 2.2 of \cite{EllisWang90} that the Hessian matrix of $\varphi_P$ is negative definite at each global maximum point of $\varphi_P$.
Now if $D\varphi_P$ vanishes at a single point $z_0$ in an open set $O$ of $\bb R^q$, if $\CH(z_0)$ is negative definite and if $f(z_0)\neq 0$, then we know by Laplace's method that
$$\int_{0} e^{n\varphi_P(z)}f(z)\,dz\sim_{n\to\infty} \frac{(2\pi)^{q/2} f(z_0)e^{n\varphi_P(z_0)}}{n^{q/2}\sqrt{|\det \CH(z_0)|}}.$$

\textbf{If $0<\beta<\beta_c$ :} according to Theorem \ref{EllisWang}, $\varphi_P$ attains its maximum at the unique point $\nu^0$ so applying Laplace's method yields
$$\munbe([\om])\sim_{n\to\8}\frac{f(\nu^0)}{1}=\frac{1}{q^p}.$$

\textbf{If $\beta>\beta_c$ :} Theorem \ref{EllisWang} states that $\varphi_P$ attains its maximum at exactly $q$ points $\nu^i(\beta)$, $i=1,\cdots,q$, where $\nu^i(\beta)$, $i=2,\cdots,q$ is obtained by interchanging the first and ith coordinates of $\nu^1(\beta)$.
Due to the symmetry of the function $\varphi_P$ it is clear that $\det \CH(\nu^i)=\det \CH(\nu^1)$,
$i=2,\cdots,q$.
Considering a family of disjoint open sets $(O_i)_{1\leq i\leq q}$ such that $O_i$ contains $\nu^i$ and $\bb R^q=\cup_{i=1}^q O_i \cup N$, where $N$ is a set of measure zero, Laplace's method yields
$$\munbe([\om])\sim_{n\to\8}\frac{1}{q}\sum_{i=1}^q f(\nu^i).$$
Recall that 
$$f(\nu^i)=\frac{e^{\beta\langle\nu^i,L_p(\om)\rangle}}{\pare{\sum_{k=1}^q e^{\be \nu^i_k}}^p}$$
with 
$$\nu^i_k=\left\{\begin{aligned}&\frac{1-s_\be}{q}& \textrm{ if } k\neq i,\\
&\frac{1+(q-1)s_\be}{q}& \textrm{ if } k=i.\end{aligned}\right.$$
As $\sum_{k=1}^q L_{p,k}(\om)=p$ it is easily seen that
\begin{equation}\label{int1}
e^{\beta\langle\nu^i,L_p(\om)\rangle}=\exp\pare{\frac{\be p(1-s_\be)}{q}+\be s_\be L_{p,i}(\om)}.
\end{equation}
As $\nu^i$ is a critical point of $\varphi_P$ and 
$\frac{\partial\varphi_P}{\partial z_i}(z)=\frac{\beta e^{\beta z_i}}{\sum_{k=1}^q e^{\be z_k}}-\beta z_i$, we know that
\begin{equation}\label{int2}
\sum_{k=1}^q e^{\be \nu^i_k}=\frac{e^{\be \nu^i_j}}{\nu^i_j}=
\frac{q}{1-s_\be}\,e^{\frac{\be(1-s_\be)}{q}}.
\end{equation}
Putting together \eqref{int1} and \eqref{int2} we obtain
\begin{equation*}
f(\nu^i)=\pare{\frac{1-s_\be}{q}}^p e^{\be s_\be L_{p,i}(\om)},
\end{equation*}
which can also be written
\begin{equation}\label{fnui}
f(\nu^i)=\frac{1}{(e^{\beta s_\be} +q-1)^{p}}\, e^{\be s_\be L_{p,i}(\om)}
\end{equation}
since $s_\be$ is solution of the equation \eqref{sbeta}.
Therefore
$$\munbe([\om])\sim_{n\to\8}\frac{1}{q}\frac{1}{(e^{\beta s_\be} +q-1)^{p}}\sum_{i=1}^q e^{\be s_\be L_{p,i}(\om)}.$$

\textbf{If $\beta=\beta_c$ :} the function $\varphi_P$ admits exactly $q+1$ maximun points $\nu^i(\beta)$, $i=0,\cdots,q$ but $\det \CH(\nu^0)\neq\det \CH(\nu^1)$,
therefore Laplace's method yields
\begin{equation}\label{munbec}
\munbe([\om])\sim_{n\to\8}\frac{|\det \CH(\nu^0)|^{-1/2}f(\nu^0)+
|\det \CH(\nu^1)|^{-1/2}\sum_{i=1}^q f(\nu^i)}{|\det \CH(\nu^0)|^{-1/2}+
q \;|\det \CH(\nu^1)|^{-1/2}}.\end{equation}
In the proof of Proposition 2.2 of \cite{EllisWang90} it is proved that $\CH(\nu^0)$
has a simple eigenvalue at $\be$ and an eigenvalue of multiplicity $(q-1)$ at $\be q^{-1}(q-\be)$ whereas $\CH(\nu^1)$
has simple eigenvalues at $\be$ and $\beta-\beta^2 qab$ and an eigenvalue of multiplicity $(q-2)$ at $\beta-\beta^2 b$, where $a=q^{-1}(1+(q-1)s_\be)$ and $b=q^{-1}(1-s_\beta)$. Recalling that $s(\beta_c)=\frac{q-2}{q-1}$ we deduce that
$$|\det \CH(\nu^0)|=\be_c^q(1-q^{-1}\beta_c)^{q-1},$$
$$|\det \CH(\nu^1)|=\be_c^q(1-q^{-1}\beta_c)\pare{1-\frac{\be_c}{q(q-1)}}^{q-2}.$$
Reporting in \eqref{munbec} and recalling \eqref{fnui} we get the result.
\end{proof}

\subsection{Identification of the limit}
We can already deduce from Lemma \ref{lemCWP} that $\mu_{n,\beta}\stackrel[n\to+\8]{w}{\longrightarrow}\wt\mu_{0}$ if $\be<\be_c$.

\begin{lemma}\label{lem-muk}\emph{\textbf{Computation for $\wt{\mu}_{b}^{k}$}}

For $k=1,\ldots,q$,
\begin{equation}\label{eq-muk}
\wt{\mu}_{b}^{k}([\om])=\frac{e^{b L_{p,k}(\om)}}{(e^{b}+q-1)^{p}}.
\end{equation}
\end{lemma}
\begin{proof}
The function $b \BBone_{[\theta^k]}$ depends only on the zero coordinate, therefore the supremum in \eqref{eq-def-pressure} is attained for the product measure $(m^k)^{\otimes\bb N}$,
where the probability vector $(m^k_j)_{1\leq j\leq q}$ on $\Lambda$ maximizes the quantity
$$-\sum_{j=1}^q p_j\log p_j+b p_k$$ over all the probability vectors $(p_j)_{1\leq j\leq q}$ on $\Lambda$, and is given by $m^k_k=\frac{e^b}{e^b +q-1}$, $m^k_j=\frac{1}{e^b +q-1}$ if $j\neq k$ (see for instance Example 4.2.2 of \cite{Keller}).
The result is then clear.
\end{proof}

The limit in \eqref{limCWP} is now a direct consequence of the lemmas \ref{lemCWP} and 
\ref{lem-muk}.

\bibliographystyle{plain}
\bibliography{biblioIsing}


\end{document}